\documentclass[11pt]{article}

\usepackage{pdfpages}
\usepackage{array}
\usepackage{amsmath}
\usepackage{amssymb}
\usepackage{pstricks}
\usepackage{amsfonts}
\usepackage{latexsym}
\usepackage{algorithm}
 \usepackage{hyperref}
\usepackage[noend]{algorithmic}
\usepackage{todonotes}

\newcommand{\E}{\mathbb{E}}
\newcommand{\R}{\mathbb{R}}
\newcommand{\N}{\mathbb{N}}

\newcommand{\D}{\mathcal{D}}

\newcommand{\ac}{\`{a} }

\newtheorem{theorem}{Theorem}
\newtheorem{proposition}{Proposition}
\newtheorem{lemma}{Lemma}
\newtheorem{corollary}{Corollary}

\newtheorem{remark}{Remark}

\def\fdimo{~\hfill$\Box$\medbreak}

\begin{document}
	
\title{Analysis of multivariate symbol statistics\\
	in primitive rational models}

\author{Massimiliano Goldwurm$^{(1)}$, Claudio Macci$^{(2)}$, Marco Vignati$^{(1)}$, Elena Villa$^{(1)}$}

\date{ }
\maketitle

\begin{center}
{\small (1) Dipartimento di Matematica, Universit\ac degli Studi di Milano, Italy\\
	\small (2) Dipartimento di Matematica, Universit\ac di Roma Tor Vergata, Italy}
\end{center}

\date{ }
\maketitle

\begin{abstract}
We study the asymptotic behaviour of sequences of multivariate random variables
representing the number of occurrences of a given set of symbols in a word of length $n$ 
generated at random according to a rational stochastic model.
Assuming primitive the matrix of the total weights of transitions of the model,
we first determine asymptotic expressions for the mean values and the covariances of such statistics.
Then we establish two asymptotic results that generalize known univariate cases to different regimes:
a large deviation principle with speed $n$, implying almost sure convergence, and a multivariate 
Gaussian limit. Additionally, we introduce a novel moderate deviation result as a bridge between these
regimes. Central to our proofs is a quasi-power property for the moment generating function of the statistics, allowing us to employ the Gärtner-Ellis Theorem for both large and moderate deviations.
\end{abstract}



\section{Introduction}
\label{sec:introd}

This work is devoted to the analysis of sequences of multivariate random variables,
representing the number of occurrences of symbols, belonging to a given finite
alphabet $\Sigma$, in a word of length $n$ in $\Sigma^*$
(\footnote{Traditionally, $\Sigma^*$ is the set of all words over $\Sigma$.}), 
generated at random according to a rational stochastic model.
This model was introduced in \cite{bcgl03} to study the number of occurrences of patterns in words of given length in a regular language.
It can be defined by a (nondeterministic) finite state automaton,
equipped with positive real weights
on transitions (called \emph{generalized automaton} in \cite{tu69}). 
For any $n\in\N_+$, the probability of a word $w\in \Sigma^*$ of length $n$ is here
proportional to the total weight of the accepting computations of the automaton
on input $w$.
These stochastic models include as special case the uniform random generation of words
of given length in any regular language.
They also include other classical probabilistic models used in the analysis of
pattern statistics to generate at random a text of given length, as the Bernoullian or
the Markovian models \cite{nsf02,rs98,fsv06,js15}.

The rational model in the univariate case, when $\Sigma$ is reduced to a
binary alphabet, has been studied in several previous papers and various properties
have been obtained for the number of occurrences of a single symbol in a random
word of given length.
In particular, when the matrix of the weights of the transitions is primitive, 
such a quantity has a Gaussian limit distribution, it owns various properties of
local limit and enjoys a large deviations principle \cite{bcgl03,dgl04,glv23,gv25}.

The main goal of the present work is to extend some of those results
to the multivariate case, when the alphabet $\Sigma$ contains $\ell +1$ elements
and the random variable under investigation has $\ell > 1$ components.
Here in particular we are interested in the large deviation properties recently 
studied in \cite{gv25,gv25a} in the univariate case.
We recall that the properties of large deviations are a classical subject of probability
\cite{dz10,dh00}, often studied in the context of pattern statistics \cite{br14,dr04},
and more generally in the analysis of combinatorial structures \cite{hw96,fs09}.

To present our results, let us denote by ${\bf Y}_n$ the random variable 
of $\ell$ components,
representing the number of occurrences of the first $\ell$ elements of $\Sigma$ in
a word of length $n$ generated at random according to a rational stochastic model.
Assuming primitive the matrix of the weights of the automaton that defines the model,
we determine asymptotic expressions for the mean values and the covariance matrix of
$Y_n$.
In particular we have $\E[{\bf Y}_n] = \boldsymbol{\beta} n + o(n)$ for some
$\boldsymbol{\beta} \in \R^{\ell}$.
Moreover, under the same primitivity hypothesis, we prove the following results:
\begin{itemize}
	\item $\frac{{\bf Y}_n - \boldsymbol{\beta}n}{\sqrt{n}}
	=\sqrt{n}\left(\frac{{\bf Y}_n}{n}-\boldsymbol{\beta}\right)$ converges in distribution to a 
	centered Gaussian random variable;
	\item $\{{\bf Y}_n/n\}$ satisfies a Large Deviation Principle with speed $n$, implying that the probabilities
	$P(\| {\bf Y}_n/n - \boldsymbol{\beta}\| \geq \delta)$ decay exponentially to zero (for every $\delta>0$), and
	therefore ensuring the almost sure convergence of ${\bf Y}_n/n$ to $\boldsymbol{\beta}$;
	\item a moderate deviation result holds, providing a bridge between the two aforementioned asymptotic regimes.
\end{itemize}
We note that, while the first two findings are consistent with known results for the univariate case $\ell=1$ 
\cite{bcgl03,gv25}, the moderate deviation result represents a entirely new contribution.
Both large and moderate deviations are established by applying the G\"{a}rtner-Ellis Theorem, made possible by the 
quasi-power property of the moment generating function of ${\bf Y}_n$.

The material we present is organized as follows: in Section \ref{sec:prelim} we recall standard notions
and results concerning the large deviations; in Section \ref{sec:mratmod} we define the rational model
in our context and introduce the random variable ${\bf Y}_n$;
in Section \ref{sec:primitive} we give some properties of ${\bf Y}_n$ in the primitive case,
in particular we show a quasi-power theorem for its moment generating function, and
prove the asymptotic expression of its mean values and covariance matrix.
The Gaussian limit distribution for $\frac{{\bf Y}_n - \boldsymbol{\beta}n}{\sqrt{n}}$ is presented in
Section \ref{sec:normalap}.
The other two results mentioned above are then proved in Sections \ref{sec:largedev} and
\ref{sec:moderatedev}, respectively.
Section \ref{sec:conclusioni} summarizes the results with some further comments and discusses possible open problems
while, in the final appendix (Section \ref{sec:appendice}), we recall some known properties that may be useful for the reader.

\section{Preliminaries on large (and moderate) deviations}
\label{sec:prelim}

In this section we recall some preliminaries on large (and moderate) deviations, 
and we refer to sequences of random variables 
taking values in $\R^m$, for any integer $m\geq 1$.
The main references are \cite{dz10,dh00}.

We start with some preliminary notation.
For any pair of elements ${\bf x}=(x_1,\ldots,x_m), {\bf y}=(y_1,\ldots,y_m) \in \R^m$,
we denote by ${\bf x}\cdot {\bf y}$ the scalar product $\sum_{i=1}^n x_i y_i$.
When necessary, an element ${\bf x} \in \R^m$ is considered as a column array, 
and ${\bf x}'$ denotes its transposed (row) array:
thus, ${\bf x}\cdot {\bf y}$ can also be represented by ${\bf x}'{\bf y}$.
For every set $D\subset \R^m$, we denote by $D^c$, $D^\circ$ and $\overline{D}$, respectively,
the complement, the interior and the closure of $D$.
Moreover, for any function $f:\R^m \rightarrow [-\infty, \infty]$, we set 
$\D_f:=\{ {\bf x}\in\R^m : f({\bf x}) \in \R \}$,
$\nabla f({\bf x})$ represents the gradient of $f$ at ${\bf x}\in\R^m$,
and $H f({\bf x})$ is the Hessian matrix of $f$ at ${\bf x}\in\R^m$.
At last, for every $\delta\geq 0$ and every ${\bf x_0} \in \R^m$,
$B_\delta({\bf x_0}) := \{{\bf x} \in \R^m : \|{\bf x} - {\bf x_0}\| < \delta\}$ 
is the open ball centered in $\bf{x}_0$ having radius $\delta$.

A \emph{rate function} is defined as a lower semi-continuous map $I: \R^m \rightarrow [0,\infty]$;
such a function is said to be \emph{good} if, for any $c\geq 0$,
the level set $\{{\bf x}\in\R^m:I({\bf x})\leq c\}$ is compact.
We recall that every lower semi-continuous function always attains a minimum value in any nonempty 
compact set.

Now, consider a sequence $\{v_n\}_{n\geq 1}$ of positive reals such that $\lim_{n\to\infty}v_n=\infty$.
We say that a sequence of $\R^m$-valued random variables
$\{{\bf X}_n\}_{n\geq 1}$, defined on the same probability space $(\Omega,\mathcal{F},P)$, satisfies the Large Deviation Principle (LDP for short)
with speed $v_n$ and rate function $I$, if the following relations hold:
\begin{gather}
\label{ldpaperti}
\liminf_{n\to\infty}\frac{1}{v_n}\log P({\bf X}_n\in O)\geq -\inf_{{\bf x}\in O}I({\bf x}) \ ,  \qquad
\mbox{ for all open sets $O\subseteq \R^m$;}\\
\label{ldpchiusi}
\limsup_{n\to\infty}\frac{1}{v_n}\log P({\bf X}_n\in C)\leq-\inf_{{\bf x}\in C}I({\bf x}) \ , \qquad 
\mbox{ for all closed sets $C\subseteq \R^m$.}
\end{gather}
An equivalent formulation of \eqref{ldpaperti} and \eqref{ldpchiusi} is the following:
\begin{equation} -\inf_{{\bf x}\in D^\circ}I({\bf x})\leq \liminf_{n\to\infty}\frac{1}{v_n}\log P({\bf X}_n\in D)\leq \limsup_{n\to\infty}\frac{1}{v_n}\log P({\bf X}_n\in D)\leq -\inf_{{\bf x}\in \overline{D}}I({\bf x})
\label{LDP}
\end{equation}
for all $D\subset\R^m$.

From the definition of LDP given above one can deduce some results and simple observations that are useful in our context:
\begin{enumerate}
 \item If $\inf_{{\bf x}\in D^\circ}I({\bf x})=\inf_{{\bf x}\in \overline{D}}I({\bf x})$ for some $D\subset\R^m$, then 
 $$\lim_{n\to\infty}\frac{1}{v_n}\log P({\bf X}_n\in D)=-\inf_{{\bf x}\in D}I({\bf x});$$
 so, roughly speaking, $P({\bf X}_n\in D) \approx e^{-v_n\inf_{x\in D} I({\bf x})}$ (and we have an exponential decay to zero if
 $\inf_{x\in D} I({\bf x})>0$). Then, the greater the value of  $\inf_{x\in D} I({\bf x})$, the less likely the event $\{{\bf X}_n\in D \}$ 
 is to occur when $n$ is large. In particular, when $P({\bf X}_n\in D)=0$ for every $n$, then $\inf_{x\in D} I({\bf x})=+\infty$.
 \item 
 If the random variables  ${\bf X}_n$ take values in a closed set $C_\bullet$, that is $P({\bf X}_n\in C_\bullet)=1$ for every $n$,
 then the rate function $I({\bf x})$ is equal to infinity on the complement  $C_\bullet^c$.
 \item 
 It is possible to prove (see Section \ref{app convergenza}) that,
 if $I$ is a good rate function, and there exists ${\bf x}_0$ such that $I({\bf x})=0$ if and only if 
 ${\bf x}={\bf x}_0$, then the sequence $\{{\bf X}_n\}_{n\geq 1}$ converges in probability to ${\bf x}_0$. 
 Moreover if $v_n=n$ the convergence is almost sure.
 \end{enumerate}
 
In our work, the sequence $ \{\mathbf{X}_n\}_{n\geq 1}$ will be of the form  
$ \{\mathbf{Y}_n/n\}_{n\geq 1}$, where
$\mathbf{Y}_n=(Y_{n,1}, \ldots, Y_{n,m})$, each component $Y_{n,i}$
takes values in $ \{0,1,\ldots,n\} $, and we have
$Y_{n,1} + \cdots + Y_{n,m} \leq n$.
Therefore, the set $ C_\bullet $ mentioned above is the traditional simplex, i.e.
$$C_\bullet=\{(x_1,\ldots,x_m)\in\mathbb{R}^m:x_1,\ldots,x_m\geq 0,x_1+\cdots+x_m\leq 1\}.$$
We will show that the sequence $ \{\mathbf{X}_n\}_{n\geq 1}$ satisfies a LDP with speed $v_n=n$. It will also follow that the corresponding rate function vanishes only at a single point $\mathbf{x}_0$, which will be specified and denoted by 
$\boldsymbol{\beta}$.

We will also present a result on \emph{moderate deviations}.
This terminology is used in the literature to refer to a suitable class of large deviation principles
   which bridges the gap between two asymptotic regimes:
   \begin{itemize}
   	\item the convergence of $\{{\bf X}_n-{\bf x}_0\}_{n\geq 1}$ to zero;
   	\item the weak convergence of $\{\sqrt{n}({\bf X}_n-{\bf x}_0)\}_{n\geq 1} $ to the centered Normal distribution with a suitable 
   	covariance matrix $V$.
   \end{itemize}
   More precisely, we mean that, for every positive sequence $\{a_n\}_{n\geq 1}$ such that $a_n\to 0$ and 
   $na_n\to\infty$, we will establish an LDP for the sequence $\{\sqrt{na_n}({\bf X}_n-{\bf x}_0)\}_{n\geq 1}$ 
   with speed $1/a_n$ and an appropriate good rate function $J$. So, in some sense, the two asymptotic regimes
   above correspond to $a_n=1/n$ (in this case $na_n\to\infty$ fails) and $a_n=1$ (in this case $a_n\to 0$ fails),
   respectively.

The large deviation result for $ \{\mathbf{X}_n\}_{n\geq 1}$ with speed $n$,
as well as the moderate deviation result, will be obtained through suitable applications of the 
\emph{G\"{a}rtner–Ellis theorem} (see e.g. \cite[Th. 2.3.6]{dz10}), whose statement is recalled below.
To this end, we first recall that a convex function $f:\R^m\to(-\infty,\infty]$  is said to be
\emph{essentially smooth} if the interior $\mathcal{D}_f^\circ$ is
non-empty, $f$ is differentiable throughout
$\mathcal{D}_f^\circ$, and $f$ is steep, i.e.
$\lim_{n\to\infty}\|\nabla f({\bf x}_n)\|=\infty$ whenever
$\{{\bf x}_n\}_{n\geq 1}$ is a sequence in $\mathcal{D}_f^\circ$
converging to some point on the boundary of $\mathcal{D}_f$.

\begin{theorem}[G\"{a}rtner-Ellis Theorem]
\label{GET}
Let $\{{\bf Z}_n\}_{n\geq 1}$ be a sequence of $\mathbb{R}^m$-valued
random variables such that there exists the function
$\Lambda:\R^m\to(-\infty,\infty]$ defined by
$$\Lambda({\bf t}):=\lim_{n\to\infty}\frac{1}{v_n}\log\mathbb{E}\left[e^{v_n\;{{\bf t}}\cdot{\bf Z}_n}\right]\ ,
\qquad \forall\ {\bf t}\in\R^m.$$ 
Assume that the origin ${\bf 0}=(0,\ldots,0)\in\R^m$ belongs to the interior
$\mathcal{D}_\Lambda^\circ$ and also that 
 $\Lambda$ is essentially smooth and lower semi-continuous.
Then $\{{\bf Z}_n\}_{n\geq 1}$ satisfies the LDP with speed $v_n$ and good
rate function $\Lambda^*$ defined by
$$\Lambda^*({\bf x}):=\sup_{{\bf t}\in\mathbb{R}^m}\{{\bf t}\cdot{\bf x}-\Lambda({\bf t})\}.$$
\end{theorem}

\begin{remark}
\label{rem:convergence}
The function $\Lambda^*$ in Theorem \ref{GET} is said to be the Legendre-Fenchel transform of the function $\Lambda$, and
it is a convex and lower semi-continuous function. If we can apply the G\"{a}rtner-Ellis Theorem, then the rate function uniquely
vanishes at ${\bf x}=\nabla \Lambda({\bf 0})$.
This property can be proved by some standard argument of convex analysis, 
we here recall in Section \ref{app unique point}. 
Therefore, as mentioned at point 3 above, $\{{\bf Z}_n\}_{n\geq 1}$ converges in probability to $\nabla \Lambda({\bf 0})$.
\end{remark}

\begin{remark}\label{insieme rate finita}
Under the assumptions of the G\"{a}rtner-Ellis Theorem (in particular by the differentiability of $\Lambda$), 
we have
$${\bf x}\in \mathcal{D}_{\Lambda^*}^{\ o} \;\Longleftrightarrow\; \mathbf{x} \in \left\{ \nabla \Lambda(\mathbf{t}) : \mathbf{t} \in \mathcal{D}_\Lambda^{\ o} \right\} \ .
$$
\end{remark}

\begin{remark}\label{rem ess}
The covariance matrix $V$ above is the Hessian matrix of the function $\Lambda$ (for the initial LDP with speed $n$)
evaluated at ${\bf t}={\bf 0}$.

In both instances where the Gartner–Ellis Theorem will be applied later on (namely, 
for the initial large deviations principle with speed $n$, and for the moderate deviations regime), 
the function $\Lambda$ is finite and differentiable all over $\R^m$; 
hence, the steepness condition will be trivially satisfied in each case.  
\end{remark}

\section{Multivariate symbol statistics in rational models}
\label{sec:mratmod}

In this section we define the rational stochastic model used in this work.

Given a positive integer $\ell$, 
consider a finite alphabet $\Sigma =\{a_1,\ldots,a_\ell,b\}$.
As usual, $\Sigma^*$ denotes the set of all words over $\Sigma$,
$\epsilon$ is the empty word in $\Sigma^*$, $\Sigma^+=\Sigma^* \backslash \{\epsilon\}$ while,
for every $w\in \Sigma^*$ and every $x\in \Sigma$, $|w|$ is the length of $w$ and
$|w|_x$ is the number of occurrences of $x$ in $w$.
Clearly $\Sigma^*$ together with $\epsilon$ and the concatenation between words forms a monoid
(called the free monoid over $\Sigma$).
For every $n\in \N$, we also denote by $\Sigma^n$ the set $\{w\in\Sigma^* : |w|=n\}$.
Moreover, for an integer $m\geq 1$,
fix two (column) vectors $\xi, \eta \in \R_+^m$,  a matrix $B \in \R_+^{m\times m}$
and a set of matrices $A_i \in \R_+^{m\times m}$, $i=1,\ldots,\ell$, 
where $\R_+ =[0,\infty)$.

Now, let $\mu : \Sigma^* \rightarrow \R_+^{m\times m}$ be the monoid morphism generated by mapping
$\mu(b) = B$ and $\mu(a_i) = A_i$, for every $i=1,\ldots,\ell$.
Clearly $\mu(\epsilon)=I$ and, for every $w=w_1\cdots w_n \in \Sigma^+$ 
with $w_i\in\Sigma$ for each $i$, one has $\mu(w)=\mu(w_1)\cdots \mu(w_n)$.
We define the function $r : \Sigma^* \rightarrow  \R_+$ as the map associating each $w\in \Sigma^*$ 
with the value $\xi' \mu(w) \eta$.
According to the literature \cite{br88,ss78}, 
$r$ is \emph{rational} formal series in the non-commutative variables $a_1,\ldots,a_\ell,b$,
and it may be represented in form $r=\sum_{w\in\Sigma^*} (r,w) w$,
where $(r,w)$ denotes the value of $r$ at the word $w$, that is $(r,w)=\xi' \mu(w) \eta$.
We say that the triple
$(\xi,\mu,\eta)$ is a \emph{linear representation} of $r$.
Clearly, such a triple $(\xi,\mu,\eta)$ can be considered as a finite state automaton over the alphabet $\Sigma$,
with transitions weighted by positive real values.
Thus, each matrix $A_i$ (resp. $B$) represents the weights of the transitions labelled by $a_i$ (resp. $b$),
while $\xi$ (resp. $\eta$) represents the weights of the initial (resp. final) states.

Throughout this work, we assume that the set $\{w\in\Sigma^n : (r,w)>0\}$ is non-empty
for every $n\in\N_+$ large enough (so that $\xi\neq {\bf 0} \neq \eta$), and that
all $A_i$ and $B$ are non-zero matrices (i.e., each of them has at least one positive entry).
Note that, for every $n\in\N$,
one can easily compute the sum of all values of $r$ associated with words in $\Sigma^n$:
in fact, defining 
$M= A_1+\cdots+A_\ell+B$, we have
\begin{equation}
\label{sumlinexpr}
\sum_{w\in\Sigma^n} (r,w) = \xi' \sum_{w\in\Sigma^n} \mu(w) \eta =
\xi'  \left( \prod_{i=1}^n \ \sum_{w_i\in\Sigma} \mu(w_i)   \right) \eta =
\xi' M^n \eta \ .
\end{equation}
Thus, we can consider the probability measure $\mbox{Pr}$ over the set $\Sigma^n$
given by
\begin{equation}
\label{probw}
\mbox{Pr}(w): = \frac{(r,w)}{\sum_{x\in\Sigma^n} (r,x)} = \frac{\xi' \mu(w) \eta}{\xi' M^n \eta} 
\ , \qquad  \forall \ w\in \Sigma^n \ .
\end{equation}

Note that, if $r$ is the characteristic series of a language $L\subseteq \Sigma^*$,
then $\mbox{Pr}$ is the uniform probability function over the set
$L\cap \Sigma^n$.
Also observe that the traditional Markovian models (to generate a word at random in $\Sigma^*$) 
occur when $M$ is a stochastic matrix, $\xi$ is a stochastic array and $\eta'=(1,1\ldots,1)$.

Then, under the previous hypotheses, for every $n\in \N_+$ we can define the multivariate random variable 
$${\bf Y}_n = ( |w|_{a_1},  |w|_{a_2}, \ldots,  |w|_{a_\ell} ) \ ,$$
where $w$ is chosen at random in $\Sigma^n$ with probability $\mbox{Pr}(w)$.
We have ${\bf Y}_n =(Y_{n,1},\ldots,Y_{n,\ell}) \in \{0,1,\ldots,n\}^{\ell}$ and
$\sum_{i=1}^\ell Y_{n,i} \leq n$.
Thus, ${\bf Y}_n$ takes value in the set $\mbox{Sim}_n$ defined by
$$\mbox{Sim}_n := \left\{{\bf k}=(k_1,\ldots,k_\ell) \in \N^\ell : \sum_{i=1}^n k_i \leq n \right\} \ .$$
Moreover, by relations (\ref{sumlinexpr}) and (\ref{probw}) the probability function of ${\bf Y}_n$ is given by
\begin{eqnarray}
\nonumber
p_n({\bf k}) := P({\bf Y}_n ={\bf k}) & = & 
\frac{\displaystyle \sum_{|w|=n, |w|_{a_1}=k_1, \ldots,  |w|_{a_\ell}=k_\ell} (r,w)}{\displaystyle \sum_{w\in \Sigma^n} (r,w)}
\ = \\
\label{probipsilon}
& = &
\frac{[{\bf x}^{{\bf k}}]\ \xi' (A_1x_1+\cdots + A_\ell x_\ell+B)^n \eta}{\xi' M^n \eta}
\ , \qquad \forall\ {\bf k} \in \mbox{Sim}_n \ ,
\end{eqnarray}
where, for any polynomial $g({\bf x})\in \R[{\bf x}]$, we denote by 
$[{\bf x}^{{\bf k}}] g({\bf x})$ the coefficient of the monomial $x_1^{k_1}\cdots x_\ell^{k_\ell}$
of degree ${\bf k}$ in $g({\bf x})$.

Now, let us consider the moment generating function $\Psi_n({\bf t})$ of ${\bf Y}_n$,
for ${\bf t}\in \R^\ell$.
Defining
\begin{equation}
\label{accaenne}
h_n({\bf t}) := \xi' (A_1e^{t_1}+\cdots + A_\ell e^{t_\ell}+B)^n \eta \ ,
\end{equation}
and using relation (\ref{probipsilon}) we have
\begin{equation}
\label{fgenmom}
\Psi_n({\bf t}) := \E(e^{{\bf t}\cdot{\bf Y}_n}) =
\sum_{{\bf k}\in \mbox{Sim}_n} p_n({\bf k}) e^{{\bf t}\cdot{\bf k}} =  
\frac{h_n({\bf t})}{h_n({\bf 0})}\ , \qquad \ \forall\ {\bf t} \in\R^\ell \ .
\end{equation}
Let us also define
$$
\Phi_n({\bf t}) := \log \Psi_n({\bf t}) = \log \E(e^{{\bf t}\cdot{\bf Y}_n}) \ .
$$

It is well known \cite{gn97,fs09} that $\Phi_n({\bf t})$ can be used to evaluate
the mean values and the covariances of ${\bf Y}_n$.
Here, we have
\begin{eqnarray}
\label{media1}
\E[{\bf Y}_n] & := & (\E[Y_{n,1}],\ldots \E[Y_{n,\ell}]) \ = \ \nabla \log \Psi_n({\bf t})|_{{\bf t}={\bf 0}} \ = \ 
\nabla\Phi_n({\bf 0}) \ = \ \frac{\nabla h_n({\bf 0})}{h_n({\bf 0})} \ , \\
\label{hessiana}
\mbox{Cov}({\bf Y}_n) & := & 
[\text{Cov}(Y_{n,i},Y_{n,j})]_{i,j\in\{1,\ldots, \ell\}} \ = \  H \log \Psi_n({\bf t})|_{{\bf t}={\bf 0}} \ = \  
H\Phi_n({\bf 0}) \ ,
\end{eqnarray}
where, we recall, $H\Phi({\bf 0})$ is the Hessian matrix of $\Phi({\bf t})$ at the point ${\bf t}={\bf 0}$:
$$
H\Phi_n({\bf 0})=\Big[ \frac{\partial^2 \Phi_n({\bf t})}{\partial t_i\partial t_j} 
\Big|_{{\bf t}={\bf 0}}\Big]_{i,j\in\{1,\ldots, \ell\}}. 
$$
In particular, this means that 
$$
\text{Cov}(Y_{n,i},Y_{n,j})  = 
\frac{1}{h_n({\bf 0})} \frac{\partial^2 h_n({\bf t})}{\partial t_i \partial t_j}\Big|_{{\bf t}={\bf 0}} - 
\frac{1}{h_n({\bf 0})^2} \frac{\partial h_n({\bf t})}{\partial t_i}\Big|_{{\bf t}={\bf 0}} 
\frac{\partial h_n({\bf t})}{\partial t_j}\Big|_{{\bf t}={\bf 0}} \; ,
\qquad \forall i,j\in\{1,\ldots, \ell\} \ .
$$

\section{Analysis of the primitive case}
\label{sec:primitive}
In this section we study the asymptotic behaviour of $\{{\bf Y}_n\}$ when the matrix $M$ is primitive.

Recall that a matrix $T\in \R_+^{m\times m}$ is {\it primitive} if there exists a positive integer $n$
such that $T^n>0$ (i.e. all entries of $T^n$ are strictly positive).
In this case the well-known Perron-Frobenius Theorem holds.
Here, we recall such result together with some
further consequences useful in our context (see for instance \cite[Sec 1.1]{se81}).
The theorem states that any primitive matrix $T$ has a unique eigenvalue of largest modulus $\lambda$,
usually called the Perron-Frobenius eigenvalue of $T$;
moreover, $\lambda$ is real positive, it is a simple root of the characteristic polynomial of $T$
and admits strictly positive left and right eigenvectors.
Moreover, it is known that as $n\rightarrow \infty$
\begin{equation}
	\label{potenzaprim}
	T^n = \lambda^n \ u v' (1 + O(\varepsilon^n)) \ ,
\end{equation}
for some $\varepsilon \in (0,1)$, where $v'$ and $u$ are, respectively, left and right eigenvectors of $T$, 
relative to $\lambda$, normalized so that $v' u =1$.
Clearly $v$ and $u$ are strictly positive arrays.
It is also proved that the matrix $u v'$ satisfies the identity 
\begin{equation}
	\label{automatrice}
	u v' = \frac{\mbox{Adj}(\lambda I - T)}{\phi'(\lambda)} \ ,
\end{equation}
where $\phi(x) = \det(xI - T)$, and further $\phi'(\lambda) > 0$.

Now, let ${\bf Y}_n$ be defined by a linear representation $(\xi,\mu,\eta)$ 
and assume that the matrix $M=A_1+\cdots+A_\ell+B$ is primitive.
Denote by $\lambda$ the  Perron-Frobenius eigenvalue of $M$,
and let $v'$ and $u$ be strictly positive left and right eigenvectors of $M$, relative to $\lambda$,
normalized so that $v' u =1$.
Then, by relation (\ref{potenzaprim}), we have
$M^n = \lambda^n \ u v' (1 + O(\varepsilon^n))$,
for some $\varepsilon \in (0,1)$, and also the matrix $uv'$ can be written as in
(\ref{automatrice}).

The asymptotic properties of $\{{\bf Y}_n\}$ depend on the multivariate 
function $y({\bf t})$ defined as follows.
For every ${\bf t}\in\R^\ell$, consider the matrix $M({\bf t})$ given by
\begin{equation}
\label{emmedit}
M({\bf t}) = A_1e^{t_1}+A_2e^{t_2}+\cdots+A_\ell e^{t_\ell}+B \ ,
\end{equation}
and let $y = y({\bf t})$ be the implicit function defined by the equation
$$\mbox{det}(Iy-M({\bf t})) = 0 \ ,$$
with initial condition $y({\bf 0})=\lambda$.
Since $M({\bf t})$ is primitive for every ${\bf t}\in\R^\ell$, 
$y({\bf t})$ is its Perron-Frobenius eigenvalue, and hence $y({\bf t}) > 0$ for all ${\bf t}\in\R^\ell$.
Thus, $y({\bf t})$ is a positive function, well-defined and analytic all over $\R^\ell$.

Moreover, relations (\ref{potenzaprim}) and (\ref{automatrice}) can be applied to $M({\bf t})$,
and this allows us to prove the following ``quasi power'' condition 
on the moment generating function $\Psi_n({\bf z})$, which is a key property in our context.

\begin{theorem}
\label{th:quasipower}
Given a linear representation $(\xi,\mu,\eta)$ as defined in Section \ref{sec:mratmod},
assume that the matrix $M=A_1+\cdots+A_\ell+B$ is primitive, 
and let $\lambda$ be its Perron-Frobenius eigenvalue.
Then, there exists an analytic strictly positive function $r: \R^\ell \rightarrow \R_+$, 
satisfying $r({\bf 0})=1$, such that, for any ${\bf t} \in \R^\ell$, we have
\begin{equation}
\label{quasipeq}
\Psi_n({\bf t}) = r({\bf t}) \left( \frac{y({\bf t})}{\lambda} \right)^n \left( 1 + O(\varepsilon_{\bf t}^n) \right) \ ,
\end{equation}
for some $\varepsilon_{\bf t} \in (0,1)$.
\end{theorem}
{\bf Proof.}
Applying relation (\ref{potenzaprim}) to the matrix $M({\bf t})$, for any ${\bf t} \in \R^\ell$, we obtain
$$
M({\bf t})^n = y({\bf t})^n \, u_{\bf t} v_{\bf t}'\, (1 + O(\varepsilon_{\bf t}^n)) \ ,
$$
for some $\varepsilon_{\bf t} \in (0,1)$,
where $v_{\bf t}'$ and $u_{\bf t}$ are strictly positive left anf right eigenvectors of $M({\bf t})$,
relative to $y({\bf t})$, such that $v_{\bf t}' u_{\bf t} = 1$.
This implies
\begin{equation}
\label{accanditas}
h_n({\bf t}) = \xi' M({\bf t})^n \eta =  y({\bf t})^n \, \xi'u_{\bf t} v_{\bf t}'\eta \, (1 + O(\varepsilon_{\bf t}^n))\ ,
\end{equation}
and hence
$$
\Psi_n({\bf t}) = \frac{h_n({\bf t})}{h_n({\bf 0})} = \left( \frac{y({\bf t})}{\lambda} \right)^n
\, r({\bf t}) \, (1 + O(\varepsilon_{\bf t}^n)) \ ,
$$
where $r({\bf t}) = \frac{\xi'u_{\bf t} v_{\bf t}'\eta}{\xi' u v '\eta}$.
Clearly $r({\bf 0})=1$.
Moreover, applying equation (\ref{automatrice}) to $M({\bf t})$, 
we have that all entries of $u_{\bf t} v_{\bf t}'$ are analytic functions,
which proves that also $r({\bf t})$ is analytic and strictly positive all over $\R^\ell$.
\fdimo

\bigskip
Now, let us evaluate mean value, variances and covariances of $Y_{n,i}$, for $i\in\{1,\ldots,\ell\}$,
when the matrix $M$ is primitive.
Here and in the sequel, for any pair of indices $i,j$, we assume
$y'_i({\bf t}) := \partial_i y({\bf t}) = \frac{\partial y({\bf t})}{\partial t_i}$ and
$y''_{ij}({\bf t}) := \partial_{ij} y({\bf t}) = \frac{\partial^2 y({\bf t})}{\partial t_i \partial t_j}$,
and use the same notation for $r({\bf t})$.
In the univariate case, i.e. for $\ell=1$, the asymptotic properties of 
$\E({\bf Y}_n)$ and $\mbox{Var}({\bf Y}_n)$ are known when $M$ is primitive
\cite{bcgl03,dgl04} and could be applied in our context.
However, here we prefer to present a direct proof, based on Theorem \ref{th:quasipower},
which extends the analysis to the covariances, allows us to use some steps of the argument in subsequent sections.

Let us consider, for any ${\bf t} \in \R^\ell$, the generating fuction of $\{h_n({\bf t})\}$, that is
\begin{eqnarray*}
\sum_{n \geq 0} h_n({\bf t}) w^n  & = &  \xi' \left( \sum_{n \geq 0} M({\bf t})^n w^n \right) \eta =
\xi' \left( I - w M({\bf t}) \right)^{-1} \eta \\
& = & \frac{\xi' \mbox{Adj} ( I - w M({\bf t}) \eta }{\det (I - w M({\bf t}) )}
\end{eqnarray*}
Note that the roots of the denominator are the inverses of the eigenvalues of $M({\bf t})$.
Hence, we have
$$
\sum_{n \geq 0} h_n({\bf t}) w^n = \frac{p(w,{\bf t})}{\prod_{i=1}^m (1 - w \alpha_i({\bf t}))}
$$
where $p(w,{\bf t})$ is a suitable polynomial in $w$ and $e^{t_i}$ ($i=1,\ldots, m$), 
while the values $\alpha_i({\bf t})$'s are the $m$ eigenvalues of $M({\bf t})$.
By the primitivity of $M({\bf t})$, only one of them coincides with $y({\bf t}) > 0$,
while all the other are in modulus smaller than $y({\bf t})$.
Then, from the last identity above, it is clear that $h_n({\bf t})$ is a finite sum of terms of the form
$$
c n^k \alpha_i({\bf t})^n
$$
where $k\in \N$ and $c$ is a constant (possibly depending on $i$ and $k$).
As a consequence, also the partial derivatives of $h_n({\bf t})$  are sums of terms of the same form.
This implies that the quantities $O(\varepsilon_{\bf t}^n)$ appearing on the right hand side of
(\ref{accanditas}) and (\ref{quasipeq}) satisfy the relation
\begin{equation}
\label{nablaepsilon}
\nabla O(\varepsilon_{\bf t}^n) = O(\varepsilon^n) 
\end{equation}
for some $\varepsilon \in (0,1)$ possibly depending on ${\bf t}$.

\subsection{Mean values}
From equation (\ref{media1}) we know that
\begin{equation}
\label{attesa Y_n_bis}
\E[{\bf Y}_n] = \nabla\log \Psi_n({\bf t})_{|{\bf t}={\bf 0}}  = \nabla\log\E[e^{{\bf t}\cdot {\bf Y}_n}]_{|{\bf t}={\bf 0}}.
\end{equation}
Applying the quasi-power equation (\ref{quasipeq}) and taking into account relation (\ref{nablaepsilon}),
we obtain
\begin{multline}
\E[{\bf Y}_n] = \nabla\log\Psi_n({\bf t})_{|{\bf t}={\bf 0}} = \nabla\log \Big( r({\bf t}) \left( \frac{y({\bf t})}{\lambda} \right)^n \left( 1 + O(\varepsilon_{\bf t}^n) \right)\Big)_{|{\bf t}={\bf 0}} \\
\label{derivata parziale}
= \frac{\nabla r({\bf 0})}{r({\bf 0})} + n\frac{\nabla y({\bf 0})}{y({\bf 0})}+\nabla\log \left( 1 + O(\varepsilon_{\bf t}^n) \right)_{|{\bf t}={\bf 0}} = 
\boldsymbol{\beta} n + {\bf c}+O(\varepsilon^n),
\end{multline}
which means that
\begin{equation}
\label{attesa Y_n}
\E[{\bf Y}_n] = \boldsymbol{\beta} n + {\bf c} + O(\varepsilon^n)\ , \mbox{ where  }\ 
\boldsymbol{\beta} = \frac{\nabla y({\bf 0})}{\lambda}
\mbox{ and  }\ {\bf c} = \nabla r({\bf 0}).
\end{equation}

Moreover, from results appearing in \cite{bcgl03,dgl04}, it is known that
$y'_i({\bf 0}) = v A_i u$, for every $i\in\{1,\ldots,\ell\}$.
This implies
\begin{equation}
\label{betai}
\beta_i = \frac{y'_i({\bf 0})}{\lambda} = \frac{v'A_i u}{\lambda} \ ,\quad \mbox{ and hence $0 < \beta_i < 1$
and } 0<\sum_{i=1}^\ell \beta_i < 1 .
\end{equation}
Applying the same reasoning to $M({\bf t})$, we get that also the partial derivatives $y'_i({\bf t})$
are strictly positive, i.e.
\begin{equation}
\label{derpyt}
y'_i({\bf t}) = v'_{\bf t} A_i e^{t_i}u_{\bf t} > 0
\end{equation}
for every $i\in\{1,\ldots,\ell\}$ and every ${\bf t}\in \R^\ell$.

\subsection{Covariances}
From equation (\ref{hessiana}) we know that, for every $i,j \in \{1,\ldots,\ell\}$, one has
\begin{eqnarray*}
\mbox{Cov}(Y_{n,i},Y_{n,j}) & = &
 \E[Y_{n,i} Y_{n,j}]-  \E[Y_{n,i}]\E[Y_{n,j}] = \\
& = & \partial_{ij}  \log \E[e^{{\bf t}\cdot {\bf Y}_n}]|_{{\bf t}={\bf 0}} = 
 \partial_{ij} \log \Psi_n({\bf t})|_{{\bf t}={\bf 0}} \ .
\end{eqnarray*}
Then, applying the quasi-power relation  \eqref{quasipeq}, together with \eqref{derivata parziale} and \eqref{nablaepsilon},
we obtain
$$
\partial_{ij}\log\Psi_n({\bf t})=\partial_j\Big(\frac{ \partial_i r(\bf t)}{r({\bf t})} +  n\frac{ \partial_i y(\bf t)}{y({\bf t})} + O(\varepsilon^n)\Big) \ ,
$$
and hence
$$
H\log \Psi_n({\bf t})=H\log r({\bf t})+n H\log y({\bf t}) + O(\varepsilon^n) 
$$
As a consequence,
$$
\mbox{Cov}({\bf Y}_n) = n H\log y({\bf 0}) + H \log r({\bf 0})  + O(\varepsilon^n) = 
n\Gamma + C + O(\varepsilon^n) \ ,
$$
where the matrices $\Gamma =[\gamma_{ij}]$ and $C = [c_{ij}]$ are defined by
\begin{eqnarray}
\label{gammaij}
\gamma_{ij} & = & \partial_{ij}\log y({\bf t})_{|{\bf t}={\bf 0}} =
\frac{\partial_{ij}y({\bf t})\ y({\bf t})-\partial_iy({\bf t})\partial_jy({\bf t})}{y({\bf t})^2}_{|{\bf t}={\bf 0}} = 
\frac{y''_{ij}({\bf 0})}{\lambda} - \frac{y'_i({\bf 0}) y'_j({\bf 0})}{\lambda^2}\ , \\
\nonumber
c_{ij} & = & r''_{ij}({\bf 0}) - r'_i({\bf 0}) r'_j({\bf 0}) \ .
\end{eqnarray}
In particular, the variance of each $Y_{n,i}$ is given by
\begin{eqnarray*}
\mbox{Var}(Y_{n,i}) & = & \gamma_{ii} n + c_{ii} + O(\varepsilon^n) \ , \ \ \mbox{where}\\
\gamma_{ii} & = & \frac{y''_{ii}({\bf 0})}{\lambda} - \left( \frac{y'_i({\bf 0})}{\lambda} \right)^2 \ \ 
\mbox{and}\ \ c_{ii} = r''_{ii}({\bf 0}) - (r'_i({\bf 0}))^2 \ .
\end{eqnarray*}
Moreover, applying \cite[Theorem 3]{bcgl03} in our context, it is easy to see that $\gamma_{ii} > 0$ 
for every $i\in \{1,\ldots \ell\}$.
This means that, in our hypotheses, $\mbox{Var}(Y_{n,i}) = \Theta(n)$ for each $i$.

Clearly from the previous equalities we obtain the following limits:
\begin{equation}
\label{covgamma}
\mbox{Cov}\left(\frac{Y_{n,i}}{\sqrt{n}}, \frac{Y_{n,j}}{\sqrt{n}}\right) \to \gamma_{ij}\ ,\quad  \mbox{for } n\to\infty\ ,\end{equation}
and in particular 
\begin{equation}
\label{vargamma}
\mbox{Var}\left(\frac{Y_{n,i}}{\sqrt{n}}\right) \to \gamma_{ii}\ ,\quad  \mbox{for } n\to\infty \ .
\end{equation}

\section{Multivariate Gaussian limit}
\label{sec:normalap}

Here we prove that, still under primitive hypotheses,
the random variable $\frac{{\bf Y}_n-n\boldsymbol{\beta}}{\sqrt{n}}$ has a (multivariate) Gaussian limit distribution.
The proof is based on the study of its moment generating function and the quasi-power property presented
in Theorem \ref{th:quasipower}.

First, let us fix some notation.
For any array ${\bold m} \in \R^\ell$ and any symmetric, positive semidefinite matrix $C \in \R^{\ell\times\ell}$,
let $N({\bold m}, C)$ be a multivariate Gaussian r.v. of $\ell$ components, having mean values ${\bold m}$ and
covariance matrix $C$.
We recall \cite{gn97} that the moment generating function of $N({\bold m}, C)$ is
$$
\Phi_{{\bf m} C}({\bold t}) := \exp\left\{{\bold m}' {\bold t} + \frac{1}{2} {\bf t}'C {\bold t} \right\}\ ,
\qquad {\bold t}\in \R^\ell \ .
$$

\begin{theorem}
\label{th:normalap}
Let $(\xi,\mu,\eta)$ be a linear representation such that the matrix $M=A_1+\cdots+A_\ell+B$ is primitive,
and let $\lambda$ be its Perron-Frobenius eigenvalue.
Also, let $\boldsymbol{\beta}$ and $\Gamma$ be defined as in the previous section.
Then $\frac{{\bf Y}_n-n\boldsymbol{\beta}}{\sqrt{n}}$ converges in distribution to a random variable
$N({\bold 0}, \Gamma)$.
\end{theorem}
{\bf Proof.}
First note that $\Gamma$ is symmetric and positive semidefinite, since it is the limit of covariance matrices
(by relations (\ref{covgamma}) and (\ref{vargamma}) ).

Then, let us denote by $M_n({\bf t})$ the moment generating function of $\frac{{\bf Y}_n-n\boldsymbol{\beta}}{\sqrt{n}}$,
that is
$$
M_n({\bf t}) := \E\Big[ \exp\Big\{ {\bf t}' \frac{{\bf Y}_n-n\boldsymbol{\beta}}{\sqrt{n}} \Big\}\Big] \ .
$$
The result follows by showing that, for $n \rightarrow \infty$, $M_n({\bf t})$ converges to 
$\exp\left\{\frac{1}{2} {\bf t}'\Gamma {\bold t} \right\}$ for any ${\bf t}\in \R^\ell$
(see for instance \cite{gn97,fs09}).
To this end, simply applying Theorem \ref{th:quasipower}, we have
\begin{eqnarray*}
M_n({\bf t}) & = & \Psi\left( \frac{\bf t'}{\sqrt{n}} {\bf Y}_n \right) \cdot \exp\{-{\bf t'}\boldsymbol{\beta}\sqrt{n}\} \\
& = & r({\bf t}/\sqrt{n}) \left( \frac{y({\bf t}/\sqrt{n})}{\lambda} \right)^n \left( 1 + O(\varepsilon_{{\bf t}/\sqrt{n}}^n) \right)
\exp\{-{\bf t'}\boldsymbol{\beta}\sqrt{n}\} \\
& = & \exp \left\{ \log r({\bf t}/\sqrt{n}) + n \log \frac{y({\bf t}/\sqrt{n})}{\lambda} + O(\varepsilon_{{\bf t}/\sqrt{n}}^n)
- {\bf t'}\boldsymbol{\beta}\sqrt{n} \right\} \ .
\end{eqnarray*}
Recalling that $r({\bf t})$ is analytic at ${\bf 0}$, $r({\bf 0})=1$ and $\boldsymbol{\beta}= \nabla y({\bf 0})/\lambda$, 
by suitable expansions at the first and second order
of the previous terms, we obtain
\begin{eqnarray*}
M_n({\bf t}) & = & \exp \left\{ \frac{{\bf t}'}{\sqrt{n}} \nabla r({\bf 0}) + O(1/n) +
n \left[ \frac{{\bf t}'}{\sqrt{n}} \frac{\nabla y({\bf 0})}{\lambda} + 
\frac{1}{2} \frac{{\bf t}'}{\sqrt{n}} \Gamma \frac{{\bf t}'}{\sqrt{n}} + O(n^{-3/2}) \right] - {\bf t'}\boldsymbol{\beta}\sqrt{n}      \right\} \\
& = & 
\exp \left\{ \frac{1}{2} {\bf t}' \Gamma {\bf t} + O(1/\sqrt{n}) \right\} ,
\end{eqnarray*}
which concludes the proof.
\fdimo

\section{Large deviation results}
\label{sec:largedev}

In this section we prove that, if the matrix $M$ is primitive, then the sequence $\{{\bf Y}_n/n\}$ satisfies a LDP with speed $n$.
This implies that the same sequence converges to $\boldsymbol{\beta}$ almost surely.

\begin{theorem}
\label{th:largedev}
Let $(\xi,\mu,\eta)$ be a linear representation such that the matrix $M=A_1+\cdots+A_\ell+B$ is primitive,
and let $\lambda$ be its Perron-Frobenius eigenvalue.
Then, the sequence $\{{\bf Y}_n/n\}$ satisfies a LDP with speed $n$ and good rate function
\begin{equation}
\label{ratefun}
G^*({\bf x}) := \sup_{{\bf t}\in \R^\ell} \left\{ {\bf t}\cdot {\bf x} - \log\frac{y({\bf t})}{\lambda} \right\}.
\end{equation}
Moreover, $\{{\bf Y}_n/n\}$ converges to $\boldsymbol{\beta}$ almost surely.
\end{theorem}
{\bf Proof.}
In order to apply Theorem \ref{GET}, let us consider the function 
$$
G({\bf t}) := \lim_{n\to\infty}\frac{1}{n}\log\mathbb{E}\left[e^{n \frac{{\bf t}\cdot{\bf Y}_n}{n}}\right]\ ,
\qquad \forall\ {\bf t}\in\R^\ell\ ,
$$
By Theorem \ref{th:quasipower} we have:
\begin{eqnarray*} 
G({\bf t}) & = & \lim_{n\to\infty}\frac{1}{n}\log\Psi_n({\bf t})\\
&=& \lim_{n\to\infty}\frac{1}{n}\log\Big(r({\bf t})\Big(\frac{y({\bf t})}{\lambda} \Big)^n(1+O(\varepsilon^n_{{\bf t}}))  \Big)\\&=& \lim_{n\to\infty}\Big(\frac{\log r({\bf t})}{n}+ \log\frac{y({\bf t})}{\lambda}+ \frac{\log(1+O(\varepsilon_{\bf t}^n))}{n}\Big)\nonumber\\&=&\log\frac{y({\bf t})}{\lambda}=\log\frac{y({\bf t})}{y({\bf 0})}.
\end{eqnarray*}
It is immediate to verify that  $G$ is finite-valued and differentiable all over $\R^\ell$.
Therefore, as observed in Remark \ref{rem ess}, $G$ is essentially smooth and, by Theorem \ref{GET},
we can state that sequence $\{{\bf Y}_n/n\}$ satisfies a LDP with speed $n$ 
and good rate function  $G^*({\bf x})$ given by (\ref{ratefun}).\\
Now, observe that relation \eqref{attesa Y_n} implies 
$\nabla G({\bf 0}) = \frac{\nabla{\bf y}({\bf 0})}{\lambda}=\boldsymbol{\beta}$.
Thus, by Remark \ref{rem:convergence} and point 3 of Section \ref{sec:prelim}, 
 $\{{\bf Y}_n/n\}$ converges to $\boldsymbol{\beta}$ almost surely because the speed of its LDP is $n$.\fdimo

As mentioned at point 2 of Section \ref{sec:prelim}, by our hypotheses on the linear representations $(\xi,\mu,\eta)$ 
(established in Section \ref{sec:mratmod}), 
the function $G^*$ given above may be finite only in the set
$$C_\bullet=\{ {\bf x}=(x_1,\ldots,x_\ell)\in\mathbb{R}^\ell : x_1,\ldots,x_\ell \geq 0, x_1+\cdots+x_\ell \leq 1\},$$
while $G^*({\bf x}) = +\infty$ for every ${\bf x} \in C_\bullet^c$.
This is also in accordance with Remark \ref{insieme rate finita}; indeed,
for every  ${\bf t} \in \R^\ell$ and each $i\in \{0,1,\ldots \ell\}$, applying (\ref{derpyt}) we have
$$
\frac{\partial G}{\partial t_i}({\bf t}) =
\frac{y'_i({\bf t})}{y({\bf t})} = 
\frac{ v'_{\bf t} A_i e^{t_i}u_{\bf t}}{ v'_{\bf t} M({\bf t}) u_{\bf t}} \ \in \ (0,1) \ .
$$
In particular, by relation (\ref{emmedit}), since $B\neq [0]$ and all components of $v_{\bf t}$ and $u_{\bf t}$ are
stricly positive, this implies
$$
0 < \sum_{i=1}^l \frac{\partial G}{\partial t_i}({\bf t}) < 1 \ , \quad \forall\ {\bf t} \in \R^\ell\ .
$$

\section{Moderate deviation results}
\label{sec:moderatedev}

In this section we prove a moderate deviation result based on the same primitive hypotheses used 
previously. This contribution is entirely original and it is applicable in the univariate scenario $\ell =1$ as a particular case. As established in Section \ref{sec:prelim}, this result serves as a mathematical bridge between two convergence behaviors: the almost sure convergence of ${\bf Y}_n / n$
to a constant, and the weak convergence of $\frac{{\bf Y}_n - \boldsymbol{\beta}n}{\sqrt{n}}$, or
$\sqrt{n}\left({\bf Y}_n / n-\boldsymbol{\beta}\right)$ in its equivalent form, to a centered Gaussian random variable.

\begin{theorem}
\label{th:moderdev}
Assume the same hypotheses of Theorem \ref{th:largedev}, and let $\boldsymbol{\beta}$ and
$\Gamma$ be defined by \eqref{betai} and \eqref{gammaij}, respectively
(i.e.  $\boldsymbol{\beta}=\frac{\nabla y({\bf 0})}{y({\bf 0})}$ and $\Gamma=H\log y({\bf 0})$).
Then, for every sequence of positive reals $\{a_n\}$ such that
$$
a_n\to 0 \qquad \mbox{and} \qquad na_n\to \infty \ ,
$$
we have that a LDP holds for $\displaystyle \frac{{\bf Y}_n-n\boldsymbol{\beta}}{\sqrt{n/a_n}}$ 
with speed $1/a_n$ and rate function 
$$
J^*({\bf x})=\sup_{{\bf t}}\Big\{{\bf t}\cdot {\bf x}-\frac{1}{2}{\bf t}'\Gamma {\bf t}\Big\}\ .
$$
\end{theorem}
{\bf Proof.}
For any $n \in \N$ and any ${\bf t}\in\R^\ell$, let
$$
J_n({\bf t}):= \frac{1}{1/a_n}\log\E\Big[ \exp\Big\{ \frac{{\bf t}'}{a_n}\frac{{\bf Y}_n-n\boldsymbol{\beta}}{\sqrt{n/a_n}} \Big\}\Big] \ .
$$
In order to evaluate $\lim_{n\to \infty} J_n({\bf t})$, we apply Theorem \ref{th:quasipower}.
For every ${\bf t}\in\R^\ell$, we get
\begin{eqnarray*}
 J_n({\bf t}) & = & a_n\log\E\Big[ \exp\Big\{ \frac{{\bf t}'({\bf Y}_n-n\boldsymbol{\beta})}{\sqrt{na_n}} \Big\}\Big] \\
 & = & a_n\Big( \log\E\Big[  \exp\Big\{\frac{{\bf t}'{\bf Y}_n}{\sqrt{na_n}}\Big\}\Big]-\frac{{\bf t}' n\boldsymbol{\beta}}{\sqrt{na_n}} \Big) \\ 
& = & a_n\Big(\log\Big[r( {\bf t}/\sqrt{na_n}) \Big( \frac{y({\bf t}/\sqrt{na_n})}{\lambda}\Big)^n\Big(1+O(\varepsilon^n_{{\bf t}/\sqrt{na_n}})\Big)  \Big] 
- \frac{{\bf t}'n\boldsymbol{\beta}}{\sqrt{na_n}} \Big)\\
& = & \underbrace{a_n \log r({\bf t}/\sqrt{na_n})}_{(*)}+\underbrace{a_n\log\Big(1+O(\varepsilon^n_{{\bf t}/\sqrt{na_n}})\Big)}_{(**)} +\underbrace{na_n\Big(\log \frac{y({\bf t}/\sqrt{na_n})}{\lambda} -\frac{{\bf t}'\boldsymbol{\beta}}{\sqrt{na_n}} \Big)}_{(***)} \ .
\end{eqnarray*}
Since $r({\bf t})$ is analytic near ${\bf 0}$ and $r({\bf 0})=1$, by the hypotheses on $a_n$ it is clear that
the term $(*)$ in the last expression converges to $0$ as $n\to \infty$.
Also $(**)$ tends to $0$, since ${\bf t}/\sqrt{na_n}$ approximates ${\bf 0}$, and hence
$(**) = O(a_n\epsilon^n)$, for some $\epsilon\in (0,1)$ close to $\epsilon_{\bf 0}$.

Now we show that $(***)$ tends to $\frac{1}{2}{\bf t}'\Gamma{\bf t}$,
by using the Mac Laurin expansion of $\log \frac{y({\bf s})}{\lambda}$ till the second order:
\begin{eqnarray*}
(***) & = & na_n\Big(\log \underbrace{\frac{y({\bf 0})}{\lambda}}_{=\lambda/\lambda=1}+  \frac{{\bf t}'}{\sqrt{na_n}}\underbrace{\frac{\nabla y({\bf 0})}{y({\bf 0})}}_{\boldsymbol{\beta}}  +\frac{1}{2}\frac{{\bf t}'}{\sqrt{na_n}} \Gamma\frac{{\bf t}}{\sqrt{na_n}}  + o\Big(\frac{1}{na_n}\Big)- \frac{{\bf t}'\boldsymbol{\beta}}{\sqrt{na_n}}\Big)\\
& = & na_n\Big(\frac{{\bf t}'\boldsymbol{\beta}}{\sqrt{na_n}}   + \frac{1}{2na_n}{\bf t}' \Gamma{\bf t} + o\Big(\frac{1}{na_n}\Big)- \frac{{\bf t}'\boldsymbol{\beta}}{\sqrt{na_n}} \Big) \\
& = & \frac{1}{2}{\bf t}'\Gamma{\bf t} + o(1) \ .
\end{eqnarray*}

Since the function 
$J({\bf t}) = \lim_{n\to\infty}J_n({\bf t}) =\frac{1}{2}{\bf t}' \Gamma{\bf t}$
is finite and differentiable for every ${\bf t}\in \R^\ell$, we can apply Theorem \ref{GET} and get the result. 
\fdimo

\begin{remark} 
If $\Gamma $ is invertible (i.e. $\text{det}(\Gamma)\neq 0$) 
then the $\sup$ in $J^*({\bf x})$ is taken for ${\bf t}$ such that 
${\bf x}=\Gamma{\bf t}$, that is in ${\bf t}=\Gamma^{-1}{\bf x}$.
Therefore, recalling that $(\Gamma^{-1})'=(\Gamma^{-1})$, we have
\begin{eqnarray*}
J^*({\bf x}) & = & (\Gamma^{-1}{\bf x})'{\bf x}-\frac{1}{2}(\Gamma^{-1}{\bf x})'\Gamma \Gamma^{-1}{\bf x}\\
 & = & {\bf x}'(\Gamma^{-1})'{\bf x}-\frac{1}{2}{\bf x}'(\Gamma^{-1})'{\bf x}\\
 & = & \frac{1}{2}{\bf x}\Gamma^{-1}{\bf x}.
\end{eqnarray*}

\end{remark}

\section{Conclusions}
\label{sec:conclusioni}

In this work we have studied the asymptotic behaviour of a sequence of multivariate
random variables $\{{\bf Y}_n\}$, representing the number of occurrences of a 
set of symbols in a word of length $n$, generated at random according to a rational
stochastic model.
We have considered the standard case when, in the model, the matrix of the total weights
of the transitions is primitive.
Under this assumption we have
$\E[{\bf Y}_n] = \boldsymbol{\beta} n + o(n)$ for some
$\boldsymbol{\beta} \in \R^{\ell}$, and we prove the following results:
\begin{itemize}
	\item $\frac{{\bf Y}_n - \boldsymbol{\beta}n}{\sqrt{n}}
	=\sqrt{n}\left(\frac{{\bf Y}_n}{n}-\boldsymbol{\beta}\right)$ converges in distribution to a 
	centered Gaussian random variable with a covariance matrix $\Gamma$;
	\item $\{{\bf Y}_n/n\}_{n\geq 1}$ satisfies a large deviation principle with speed $n$, implying the almost sure 
	convergence of ${\bf Y}_n/n$ to $\boldsymbol{\beta}$;
	\item for every positive sequence $\{a_n\}_{n\geq 1}$ such that $a_n\to 0$ and $na_n\to\infty$,\\
    $\{\sqrt{na_n}({\bf Y}_n/n-\boldsymbol{\beta})\}_{n\geq 1}$ satisfies the large deviation principle with speed $1/a_n$ and an appropriate good rate function $J$.
\end{itemize}
The main tools used in this analysis are a quasi-power theorem for the moment 
generating function of ${\bf Y}_n$, derived from properties of the primitive matrices,
and the classical G\"{a}rtner-Ellis Theorem, which well applies in our context.

For the sake of completeness, it is useful to determine when the covariance matrix $\Gamma$, mentioned above, is (strictly) positive definite. In fact, only in this case does the limiting Gaussian random variable have full support on $\R^\ell$  with a positive density. While we generally know that the diagonal elements of $\Gamma$ are strictly positive, this does not guarantee that the matrix is positive definite.

Another natural question that deserves a further study is whether the previous results
can be extended to the cases when the matrix of transition weights is not primitive 
(nor irreducible), but it has two or more (strictly connected) components.
Different limit distributions have been obtained in these cases \cite{dgl04,gl06,glv23}
and one may wonder under which conditions a large deviation principle still holds.

\section{Appendix: some known results}
\label{sec:appendice}

In this section we recall some known results concerning the large deviations, used in the previous sections,
which often are not explicitly mentioned in the literature and may be useful for the reader.
These properties concern in particular the unicity of zeros for rate functions and the convergence to a constant
either in probability or almost sure.

\subsection{Unique zero of a rate function}
\label{app unique point}  
	
\begin{proposition}
If the G\"{a}rtner-Ellis theorem applies (i.e. Theorem \ref{GET}), then the rate function $\Lambda^*$ uniquely
vanishes at ${\bf x}=\nabla \Lambda({\bf 0})$. 
\end{proposition}
{\bf Proof.}
It is easy to check that the rate function $\Lambda^*$ vanishes at
${\bf x}=\nabla \Lambda({\bf 0})$. Now we reason by contradiction, and
we assume that
$\Lambda^*$ is equal to zero at more than one point. Then, by its
convexity, the
function $\Lambda^*$ would be equal to zero on an interval; thus the
function
$\Lambda^*$ would not be strictly convex. Then, by a standard argument
of convex
analysis (see, e.g., \cite[Th. 26.3]{ro70}), the function $\Lambda$
would not be
differentiable. This is a contradiction for the hypotheses of the
G\"{a}rtner-Ellis,
and the proof is complete.
\fdimo

\subsection{Convergence to the unique zero of a rate function}
\label{app convergenza}

Let  us assume that  $I$ is a \emph{good} rate function on $\mathbb{R}^m$, and define $I(S):=\inf\{I(x):x\in S\}$
for any $S\subseteq \mathbb{R}^m$.

\begin{lemma}
\label{lem:attain}
	Let $S\subseteq \mathbb{R}^m$ be a closed set such that $I(S)<\infty$. 
	Then $I(S)$ is attained at some value $x^*\in S$, i.e., there exists $x^*\in S$ such that $I(x^*)=I(S)$.
\end{lemma}
{\bf Proof.}
	When a (finite) infimum exists, one can always construct a minimizing sequence of points; so let 
	$\{x_n\}$ be a sequence of points in $S$ such that
    $$\lim_{n\to\infty}I(x_n)=I(S).$$
    Then, for any $\varepsilon>0$, the points of the sequence eventually belong to the level set
    $$K_{S,\varepsilon}=\{x:I(x)\leq I(S)+\varepsilon\}$$
    which is compact by assumption (since $I$ is a \emph{good} rate function). Thus, eventually, the points of the sequence 
    belong to $S\cap K_{S,\varepsilon}$, which is compact (being the intersection of a closed set and a compact set).
    Therefore, we can extract a subsequence $\{x_{n_h}\}$ that converges to a point 
    $x^*\in S\cap K_{S,\varepsilon}$ (by compactness properties). We conclude the proof by verifying that  $I(x^*)=I(S)$. Indeed, $x^*\in S$ by construction; moreover, by the lower 
    semicontinuity of $I$, we have
    $$\liminf_{n_h\to\infty}I(x_{n_h})\geq I(x^*),$$
    from which it follows that $I(S)\geq I(x^*)$ (the limit of the subsequence is the same as the limit of the 
    original sequence); finally, since $x^*\in S$, we obtain $I(x^*)=I(S)$.
\fdimo

\begin{corollary}
	Suppose an LDP holds with a good rate function $I$. Then $I$ vanishes at least at one point.
\end{corollary}
{\bf Proof.}
	Consider the upper bound \eqref{ldpchiusi} from the definition of the LDP;  we get
	$$\limsup_{n\to\infty}\frac{1}{v_n}\log P(X_n\in\mathbb{\R}^m)\leq -I(\mathbb{\R}^m).$$
	Hence $0\leq -I(\mathbb{\R}^m)$, which implies $I(\mathbb{\R}^m)\leq 0$, and since the rate function $I$ is non-negative,
	we get $I(\mathbb{\R}^m)=0$. Therefore, by the previous lemma, there exists at least one point $x^*$ such that $I(x^*)=0$. 
\fdimo

We now present the result on convergence in probability to the unique zero of a good rate function.
\begin{proposition}
	Suppose we have a sequence of random variables $\{X_n\}$, defined on the same probability space 
	$(\Omega,\mathcal{F},P)$, satisfying an LDP with a good rate function $I$, which vanishes at a unique 
	point $x_0$. 
	Then $X_n$ converges in probability to $x_0$.
\end{proposition}
{\bf Proof.}
	Let $B_\delta(x_0)$ be the open ball of center $x_0$ and radius $\delta>0$. 
	We consider the upper bound \eqref{ldpchiusi} from 
	the definition of the LDP for the closed set $(B_\delta(x_0))^c$, obtaining
	$$\limsup_{n\to\infty}\frac{1}{v_n}\log P(X_n\in(B_\delta(x_0))^c)\leq -I((B_\delta(x_0))^c).$$
	Then, for every $\eta>0$, eventually we have
	$$P(X_n\in(B_\delta(x_0))^c)\leq \exp(-v_n(I((B_\delta(x_0))^c)-\eta)).$$
	We conclude by showing that $P(X_n\in(B_\delta(x_0))^c)\to 0$. This is trivially true if
	$I((B_\delta(x_0))^c)=\infty$. Otherwise, if $I((B_\delta(x_0))^c)<\infty$, by Lemma \ref{lem:attain} we know there exists 
	$x^*\in(B_\delta(x_0))^c$ such that $I(x^*)=I((B_\delta(x_0))^c)$, and $I(x^*)>0$ (since $x^*\in(B_\delta(x_0))^c$,
	and hence $x^*\neq x_0$, where $x_0$ is the unique zero of $I$). In conclusion, we have $P(X_n\in(B_\delta(x_0))^c)\to 0$
	since we can choose $\eta\in(0,I((B_\delta(x_0))^c))$.
\fdimo
It is worth noticing that, by using the same notation of the proof of the above proposition,
a standard application of Borel-Cantelli lemma implies the following 
\begin{corollary}
 If $\sum_{n\geq 1} \exp(-v_n(I((B_\delta(x_0))^c)-\eta))<\infty$, then $X_n$ converges almost surely to $x_0$. 
\end{corollary}
Note that the above corollary trivially holds when $v_n=n$.

Finally, we conclude by observing that if $I$ were not a \emph{good} rate function, we would not be able to reach the same conclusion as in the proposition. To illustrate this, consider the following counterexample:
	let  $I$ of the form
	$$I(x):=\left\{\begin{array}{ll}
		\frac{1}{1+x^2}&\ \text{if}\ x\neq 0\\
		0&\ \text{if}\ x=0.
	\end{array}\right.$$
	It is a rate function (indeed, it is lower semicontinuous) vanishing only at $x_0=0$; however, for any $\delta>0$, 
	we have $I((B_\delta(x_0))^c)=0$.
	
\section*{Funding}
M.G. and M.V. acknowledge the partial support of Universit\ac degli Studi di Milano, under project PSR 2024.
C.M. and E.V. acknowledge the partial support of INdAM-GNAMPA. 
C.M. also acknowledges the partial support of MUR Excellence Department Project awarded to the Department of Mathematics, University of Rome Tor Vergata (CUP E83C23000330006), and of University of Rome Tor Vergata (CUP E83C25000630005) Research Project METRO.

\bibliographystyle{plain}

\end{document}